\documentclass[12pt]{amsart}

\usepackage{amssymb,amsfonts}
\usepackage[all,arc]{xy}
\usepackage{enumerate}
\usepackage{mathrsfs}

\newtheorem{theorem}{Theorem}[section]
\newtheorem{corollary}[theorem]{Corollary}

\newtheorem{lemma}[theorem]{Lemma}

\theoremstyle{definition}
\newtheorem{definition}[theorem]{Definition}

\newtheorem{remark}[theorem]{Remark}

\theoremstyle{remark}

\makeatletter
\let\c@equation\c@theorem
\makeatother
\numberwithin{equation}{section}

\title{ CONTINUITY OF SUBHARMONIC FUNCTIONS}

\author{MANSOUR KALANTAR}

\date{}

\begin{document}

\begin{abstract}

We prove that the set of points where a subharmonic function fails to be continuous is polar.
\end{abstract}

\maketitle

\section{Introduction}

Let $X$ be a metric space and $f$ a real-valued function on $X$. A classic theorem of topology, known as the semicontinuity lemma, states that if $f$ is upper semi-continuous, then the set of the points where $f$ is discontinuous is included in a countable union of closed nowhere dense sets (see for example K. Kuratowski \cite{Kuratowski} and  B. Santiago \cite{Santiago}). 

A natural question arises: Suppose that the function is defined on an open subset of an Euclidean space;  does its set of discontinuity gets  smaller, if  furthermore the function is subharmonic? An anologue of Lusin's theorem in potential theory states that for all arbitrary $\varepsilon>0$ there exists an open subset of capacity less than $\varepsilon$ on the complementary of which $u$ is continuous (See for example D. Armitage and S. Gardiner\cite[Theorem 5.5.8]{ArmGar}). We also know, by a theorem of Baire, that since subharmonic functions are pointwise limit of continuous functions, their discontinuity set is of the first Baire category (see O. Knill \cite{Knill} and the references therein). 

In this paper we first prove a basic potential theoretic result (Lemma 3.1)  and then as its first application show that the set of points where a subharmonic function fails to be continuous is  polar. Then we show that conversely,  given any polar set, there exists a suharmonic function that is discontinuous at each point of this set.  As a second application of the lemma we give a sort of "converse" for the famous extend maximum (or Phragm\`{e}n-Lindel\"{o}f) principle.

\section{Definitions and Preliminaries} 

\textit{In all this paper $\Omega$ is a bounded open subset of $\mathbb{R}^{N}$ with $N\geq2$} unless otherwise is stated. 

 We note $ \overline{E} $ and $ \partial E $ the closure and boundary of a set $ E $ in $\mathbb{R}^{N}$, respectively. A function $u:\Omega\rightarrow[-\infty,+\infty)$ is called upper semi-continuous at $x\in\Omega$ if for all $\varepsilon>0$ we can find an open neighborhood $V$ of $x$ such that 
$$u(\zeta)<u(x)+\varepsilon$$
for all $\zeta\in V$. The function $u$ is called lower semi-continuous, if $-u$ is upper semi-continuous. A function is continuous if it is lower and upper semi-continuous. 
\begin{definition}
	The set of discontinuity of a function $u$ defined on $\Omega$ is the set of points in $\Omega$ at which $u$ fails to be continuous.
\end{definition}

 For reader's convenience we summarize below the special cases of some results from the classical potential theory that we will be using. We recall definitions and results for our need in this paper,  for  the most general form the reader should  refer to the given reference.

 We recall that an upper semi-continuous function $u:\Omega\rightarrow [-\infty,+\infty)$ is called subharmonic, if $u\not\equiv-\infty$ and  for all ball $B(x,\rho)$ relatively compact in $\Omega$, 
 $$u(x)\leq \frac{1}{\sigma_{N}\rho^{N-1}}\int_{\partial B(x,\rho)}u(\zeta)d\sigma.$$

Let $\mu_{x}^{\Omega}$ be the harmonic measure at $x\in \Omega$. 
 A set $ E\subset\partial \Omega $ is called negligible for $ \Omega $, if 
 $$ \mu_{x}^{\Omega}(E)=0 $$
 for all $ x\in \Omega $. 
 Negligible sets can by characterized by the notion of thinness. A set $ E $ is said to be thin at a point $ \zeta $ if $ \zeta $ is not a fine (with respect to the fine topology) limit point of $ E. $ The following theorem, used in the proof of Lemma 3.1, gives the relation between thinness and negligibility.
 
 \begin{theorem}\label{2,2}
 	Let $ \zeta $ be a limit point of a set $ E $. The set $ E $ is thin at $ \zeta $ if and only if there exists a subharmonic function $ v $ on a neighborhood of $ \zeta $ such that 
 	$$ \limsup_{\substack{x\rightarrow\zeta\\(x\in E)}}v(x)<v(\zeta). $$
 \end{theorem}
 See \cite[ Theorem 7.2.3]{ArmGar}.

 \begin{theorem}\label{th2.3}
	\begin{itemize}	
 		\item[(i)]
 		The set $ \lbrace \zeta\in \partial \Omega: \Omega \text{ is thin at }\zeta \rbrace$ is negligible for $ \Omega $.
 		\item[(ii)]If $ E $ is a relatively open subset of $ \partial\Omega $ which is negligible, then  the set $ E $ is polar.
 	\end{itemize}
 \end{theorem}
 
 See  \cite[Theorem 7.5.4]{ArmGar} for part (i) and   \cite[Theorem 6.6.9-(i)]{ArmGar} for part (ii).

 \section{The General Lemma}

 \begin{lemma}\label{lemma3.1}
Let  $v$ be  a subharmonic function on an open neighborhood of $\overline{\Omega}$. For $\lambda\in\mathbb{R}$, define $E$ to be the set of all $x\in\Omega$ such that $v(x)<\lambda.$ Then  the set 
$$ e:=\lbrace \zeta\in\partial E: v(\zeta) >\lambda\rbrace$$
is  polar. 
 \end{lemma}
 
 \begin{proof}
  We closely follow the proof of Lemma 3.2. in author's paper \cite{Kal}. The set $E$ is open, since $v$ is  upper semi-continuous.  We start by proving that $ e $  is negligible for $ E $, i.e., its harmonic measure is zero:
\begin{equation}
\mu_{x}^{E}(e)=0 \label{3.2}
\end{equation} 
for all $ x\in e $.  To do so, it suffices to show that $E$ is thin at each point  of $ e$, according to Theorem \ref{th2.3}-(i). Let $ \zeta\in e $. It follows from the definitions of $E$ and $e$ that
$$ \limsup_{\substack{x\rightarrow\zeta\\(x\in E)}}v(x)\leq \lambda $$
whereas $$ v(\zeta)>\lambda.$$ Thus $ E$ is thin at $ \zeta $, according to  Theorem \ref{2,2} . By Theorem \ref{th2.3}-(i)  the set $ e $ is negligible  and  (\ref{3.2}) follows.

Next we proceed to prove that the set
$$\gamma:=\lbrace \zeta\in \partial E: v(\zeta)=\lambda\rbrace$$
is closed in $\partial E$. To do so it is sufficient to show that the restriction of $v$ to $\gamma$, that we steel write $v$, is continuous in the topological subspace $\partial E$. Since $v$ is already upper semi-continuous, we need to show that it is also lower semi-continuous in the mentioned topology. Let $\zeta\in \gamma$ and suppose that $v$ is not lower semi-continuous at $\zeta$. Then, there exits $\varepsilon>0$ such that for every open neighborhood $V$ of $\zeta$ there exists $x\in\partial E \cap V$ satisfying 
$$\lambda\leq v(x)\leq v(\zeta)-\varepsilon=\lambda-\varepsilon<\lambda,$$
which is absurd. Here, the first inequality is due to the fact that since $E$ is open, the boundary of $\omega$ is included to the set 
\begin{equation}
\lbrace \xi:v(\xi)\geq\lambda\rbrace.\label{eq3.4}
\end{equation}
Thus the restriction of $v$ to $\gamma$ is continuous and $\gamma$ is closed in $\partial E$.

Finally,  we have $\partial E=\gamma\cup e$ and it follows form the last paragraph that $ e=\partial E\setminus\gamma $ is open in $ \partial E $. Thus $ e $ is polar by Theorem \ref{th2.3}-(ii) and the lemma follows.

 \end{proof}

 \section{First Application :  Continuity of Subharmonic Functions}
 
  \begin{theorem}\label{th4.4}
Let  $u$ be a subharmonic function on a neighborhood of $\overline{\Omega}$.  Then the set of discontinuity of $u$ is polar, i.e. there exists a polar set $\mathcal{D}$ in $ \Omega $ such that the restriction of $u$ to $ \Omega\setminus\mathcal{D} $ is continuous.
\end{theorem}
\begin{proof}
Let $\mathcal{D}$ be the set of discontinuity of $u$. Let $\mathcal{D}_{1}$ be the subset of $\mathcal{D}$ where $u$ is real-valued, and $\mathcal{D}_{2}$ the subset of $\mathcal{D}$ where $u=-\infty$.  By definition $\mathcal{D}_{2}$ is a polar set, and we just need to prove that so is $\mathcal{D}_{1}$.

We may assume $\mathcal{D}$ non empty; otherwise there is nothing to prove. Take $x\in\mathcal{D}_{1}$. Since $u$ is subharmonic everywhere, that means $u$ is not lower semi-continuous at $x$. Thus there exists $\varepsilon>0$ such that every open neighborhood of $x$ has a point where the value of $u$ is less than or equal to $u(x)-\varepsilon$.  Let  $\lbrace x_{n}\rbrace$ be a sequence of elements of $\Omega\setminus\mathcal{D}_{2}$  that approaches $x$ as $n\rightarrow+\infty$ and such that 
\begin{equation}
u(x_{n})\leq u(x)-\varepsilon \label{315}
\end{equation} 
for all $n$.

Let $q$ be a rational number such that
\begin{equation}
u(x)-\varepsilon <q<u(x).\label{316}
\end{equation} 
We define $E_{q}$ to be the set of all $\zeta\in \Omega$ such that $u(\zeta) < q$. Thus it follows from the lemma that the set
$$e_{q}:=\lbrace \zeta\in \partial E_{q}:u(\zeta)>q\rbrace$$
is a polar set. Now, it it easy to show that the point $x$, defined in the  last paragraph as an element of $\mathcal{D}_{1}$,  belongs to $e_{j} $. To see this, first we notice that
by (\ref{315}) and (\ref{316}), we have $u(x_{n} )<q$ and so  $\lbrace x_{n}\rbrace$ is a sequence of elements of $E_{q}$. Thus its limit $x$ belongs  to the closure of $ E $, which means either to $\lbrace \zeta\in \overline{ E_{q}}: u(\zeta)\leq q\rbrace$ or to $e_{q}$. But the first option is to be ruled out by (\ref{316}); we conclude that $x$ is in the polar set $e_{j}$.

 Summing up, we have proved so far that each element $x$ in $ \mathcal{D}_{1}$ is contained in a polar set $e_{q}$. Thus 
 $$\mathcal{D}_{1}\subset \bigcup_{q\in\mathbb{Q}}e_{q},$$
where $\mathbb{Q}$ designates the set of rational numbers,  and so $\mathcal{D}_{1}$  is polar. On the other hand, the set $\mathcal{D}_{2}$ is  polar by definition. Thus the set $\mathcal{D}$ of discontinuity of $u$ is a polar set as a union of  polar sets.

\end{proof}

\subsection{The Converse Problem}

\begin{theorem}\label{th4.5}
Let  $E$ be a  subset of $ \mathbb{R}^{N}$ $(N\geq2)$. The following are equivalent.
\begin{itemize}
	\item[(i)] $E$ is  polar,
	\item[(ii)]There is a  function $u$ subharmonic on  a neighborhood of $E$ such that 
	$$E\subset \lbrace x: w(x)=-\infty\rbrace,$$
	\item[(iii)]There is a function $u$ subharmonic on a  neighborhood of $E$ such that $E$ is  the set of discontinuity of $u$.
\end{itemize}
\end{theorem}

\begin{proof}
	The equivalence between (i) and (ii) is well known. For (i) $\Rightarrow$ (iii) just set 
	\begin{displaymath}
	u(x) = \left\{
	\begin{array}{lr}
	1 & \text{if } x \in \Omega\setminus E\\
	-\infty &  \text{if } x\in E
	\end{array},
	\right.
	\end{displaymath}
	 and its converse follows from Lemma \ref{lemma3.1}.

\end{proof}

\begin{corollary}
	Let $u$ be a subharmonic function on $\Omega$. If the set of discontinuity of $u$ is not polar, then $u\equiv-\infty$. 
\end{corollary}

\begin{remark}
The set of discontinuity of a subharmonic function may very well by everywhere dense. For a function subharmonic on the unit ball $B$ of $\mathbb{R}^{N}$ that is discontinuous at each point of a dense subset of $B$, see 
 Armitage and  Gardiner \cite[Example 3.3.2]{ArmGar}.
\end{remark}

\section{Second Application: A converse for extended maximum (or Phragm\`{e}n-Lindel\"{o}f) Principle }

Recall that the extended maximum (or Phragm\`{e}n-Lindel\"{o}f)  principle provides a "negligible" set for the maximum principle of subharmonic functions. More exactly, suppose $D$ is a domain of $\mathbb{R}^{N}$ and $v$  subharmoic and bounded above on $D$ such that for some $M$ and a polar set $E\subset\partial D$ we have
$$\limsup_{\substack{x\rightarrow\zeta\\(x\in D)}}v(x) \leq M$$
for all $\zeta\in\partial D\setminus E$. Then either $u<M$ on $D$, or $u\equiv M$ on $D$ (for a more general statement and the proof see Hayman and Kennedy \cite[Theorem 5.16]{Hayman Kennedy})

The following is a sort of "converse" for the above result. It follows immediately form Lemma 3.1.

\begin{theorem}
Let $u$ be  a  subharmonic function on a neighborhood of $ \overline{\Omega}$. Suppose that there exists a set $ E $ in the boundary of $ \Omega $ such that for all $ \zeta\in\partial \Omega\setminus E $,
$$ u(\zeta)\leq M, $$
and for all $ x\in\Omega $
$$ u(x)<M. $$
Then the set 
$$ e:=\lbrace \zeta\in E:u(\zeta)>M\rbrace $$
is polar.
\end{theorem}

\end{document}